\documentclass[letterpaper,12pt]{amsart}
\usepackage{graphicx,graphics}
\usepackage{amsmath,amsthm,amssymb}

\title{Cellular resolutions of powers of monomial ideals}

\author{Alexander Engstr\"om}
\address{Department of Mathematics, Aalto University, P.O. Box 11100, FI-00076 Aalto, Finland }
\email{alexander.engstrom@aalto.fi}
\author{Patrik Nor\'en}
\address{Department of Mathematics, Aalto University, P.O. Box 11100, FI-00076 Aalto, Finland }
\email{patrik.noren@aalto.fi}

\date{\today}

\usepackage{latexsym,array,delarray,amsthm,amssymb,epsfig}

\theoremstyle{plain}
\newtheorem{theorem}{Theorem}[section]
\newtheorem{lemma}[theorem]{Lemma}
\newtheorem{proposition}[theorem]{Proposition}

\theoremstyle{definition}
\newtheorem{definition}[theorem]{Definition}

\theoremstyle{remark}
\newtheorem*{remark}{Remark}

\begin{document}

\begin{abstract}
There are many connections between the invariants of the different powers of an ideal. We investigate how to construct minimal resolutions for all powers at once using methods from algebraic and polyhedral topology with a focus on ideals arising from combinatorics. In one construction, we obtain cellular resolutions for all powers of edge ideals of bipartite graphs on $n$ vertices, supported by $(n-2)$--dimensional complexes. Our main result is an explicit minimal cellular resolution for all powers of edge ideals of paths. These cell complexes are constructed by first subdividing polyhedral complexes and then modifying them using discrete Morse theory. 
\end{abstract}

\maketitle

\section{Introduction}

The collection of all powers of an ideal contains many structures. In essence, most algebraic and homological properties stabilize after certain powers and can be derived from the smaller powers \cite{Ba, B, HH, HW, K, M, SVV}. In this paper, we study the resolutions of all powers of a monomial ideal at once. The basic philosophy is to cook up a minimal resolution that works for all powers. The monomial ideals of interest are constructed from combinatorial structures, which helps us to build the resolutions. We get the resolutions from cellular resolutions \cite{MS}, where the maps in the complex are just cellular boundary maps in a cell complex. These cell complexes are constructed from the same combinatorial data from which we define the ideals.

The cellular resolutions are constructed in several steps. First, we define polyhedral cell complexes that are very finely subdivided. For these cell complexes, it is easy to derive that they support cellular resolutions, since subcomplexes whose homology should vanish are convex. We then proceed by removing some systems of hyperplanes from the subdivision to get fewer cells. Then the subcomplexes are no longer convex, but we can use discrete Morse theory for cellular resolutions, as invented by Batzies and Welker \cite{BW}, to carry over the results of vanishing homology. So far, these resolutions are still supported by polyhedral complexes obtained by subdividing a simplex. In the next step, we turn the cellular resolutions minimal by another round of discrete Morse theory. These minimal cellular resolutions are no longer, but almost, polyhedral: the subdivided simplex only has non-polyhedral cells close to the boundary and, for large powers, most of the cellular complex is merely a subdivided simplex.

Along the way, we provide, in Proposition \ref{prop:base}, an $(n-2)$--dimensional cellular resolution for all powers of edge ideals of bipartite graphs on $n$ vertices. We then proceed to our main result in Theorem \ref{thm:main}, an explicit minimal cellular resolution of all powers of edge ideals of paths.

\section{Preliminaries}

\subsection{Monomial ideals}

The ideals resolved in this paper are monomial. In particular, we focus on those constructed from graphs. The \emph{edge ideal} of a graph $G$ with vertex set $V(G)$ and edge set $E(G)$ is the monomial ideal $\langle x_ux_v \mid uv \in E(G)  \rangle$ in $\mathrm{\bf k}[x_v \mid v\in V(G)].$ Our main example of graphs are paths. The \emph{$n$-path} $P_n$ has vertices $1,2,\ldots,n$ and edges $12,23,\ldots,(n-1)n.$

\subsection{Monomial labeling of polyhedral complexes}

\begin{definition}
Let $X$ be a cell complex. A \emph{monomial labeling of $X$} is a map $\ell$ from the set of cells of $X$ to the set of monomials in $\mathbf{k}[x_1,\ldots,x_n]$. The map $\ell$ is required to satisfy $\ell(\sigma)=\textrm{lcm}\{\ell(v)\mid  \textrm{$v$ is a vertex of $\sigma$}\}$. The $x_i$--degree of $\ell(v)$ is denoted $\ell_i(v).$
\end{definition}

A \emph{lattice point} $\alpha$ of $\mathbb{R}^n$ is a point $\alpha \in \mathbb{Z}^n.$ The \emph{monomial label} of a lattice point $\alpha$ in $\mathbb{R}_{\geq 0}^n$ is $x^\alpha.$

\begin{definition}
Let $X$ be a cell complex geometrically realized in $\mathbb{R}_{\geq 0}^n$ whose vertices are lattice points and $\ell$ be a monomial labeling giving vertices the monomial labels of their geometrical realizations. Then $\ell$ is a \emph{monomial labeling induced by the coordinates}. 
\end{definition}

In this paper, \emph{all} monomial labelings are induced by the coordinates unless stated otherwise.

\subsection{Subdivisions of Newton Polytopes}

\begin{definition}
The \emph{Newton polytope} of a monomial ideal $I$ in $\mathrm{\bf k}[x_1,\ldots x_n]$ is the polytope $\textrm{Newt}(I)$ in $\mathbb{R}^n$ spanned by the exponent vectors of the generators of $I.$
\end{definition}

Polytopes whose vertices are lattice points are \emph{lattice polytopes} and the Newton polytopes are examples of these. The \emph{$d$--dilation} of a polytope $P$ is a polytope $dP$ given by multiplying the vertex vectors of $P$ by $d.$ 

\begin{definition} Let $P$ be a lattice polytope in $\mathbb{R}^n$ with lattice points $v_1,v_2, \ldots ,v_t.$ If for all positive integers $d$, every point $p\in dP \cap \mathbb{Z}^n$ can be expressed as
$p=n_1v_1+\cdots n_tv_t$ with all $n_i \in \mathbb{Z}_{\geq 0}$, then $P$ is a \emph{normal} polytope.
\end{definition}

\begin{proposition}\label{prop:semiHibi}
If $G$ is a bipartite graph, then for all positive integers $d,$
\begin{itemize}
\item[(i)] the polytopes $\textrm{\emph{Newt}}(I_G^d)$ and $d\cdot \textrm{\emph{Newt}}(I_G)$ are the same,
\item[(ii)] the monomial labels of the lattice points of $\textrm{\emph{Newt}}(I_G^d)$ is its set of minimal generators, and
\item[(iii)] a subdivision of $\textrm{\emph{Newt}}(I_G^d)$ by integral translates of coordinate hyperplanes has only lattice points as vertices.
\end{itemize}
\end{proposition}
\begin{proof}
The mixed powers in the ideal become linear combinations in the Newton polytope. This establishes (i).

Every generator of $I_G^d$ is a monomial label of a lattice point of $\textrm{Newt}(I_G^d).$ To get the converse, we need that every lattice point can be written as an integer weight combination of the vertices of $\textrm{Newt}(I_G),$ or equivalently, that  $\textrm{Newt}(I_G)$ is normal. According to Corollary 2.3 of \cite{hibi}, 
$\textrm{Newt}(I_G)$ is normal for a class of graphs including the bipartite ones. This proves (ii).

The statement (iii) builds on Proposition 1.3 in \cite{hibi}. According to that, 
the codimension of $\textrm{Newt}(I_G)$ is two if $G$ is a connected bipartite graph 
and the two equalities cutting it out are given by $\sum_{v\in A} \alpha_v = 1$ and $\sum_{v\in B} \alpha_v = 1$ where $A$ and $B$ are the two parts of the graph.

We may assume that $G$ does not contain isolated vertices. Our proof of (iii) goes by induction on the number of edges of $G.$ The base case is clear. 

If $G$ is not connected, then we get (iii) by considering the connected components separately. For connected $G,$ we only have to consider interior vertices of the subdivision of
$\textrm{{Newt}}(I_G^d)$ because  boundary vertices are on faces that are themselves Newton polytopes of some $I_H^d,$ where $H$ is obtained from $G$ by deleting edges. And for $H$ we are already done by induction.

We are left with showing that every point given by intersecting the hyperplanes\linebreak $\sum_{v\in A} \alpha_v = d$ and $\sum_{v\in B} \alpha_v = d$ with integer translates of coordinate hyperplanes is a lattice point. Every minimal set of such hyperplanes needs to be of the form $\alpha_v=i_v\in \mathbb{Z}$ for $v\in A'\cup B'$ where $A' \subset A,$ $B' \subset B,$ and $|A \setminus A'| = |B \setminus B'| = 1.$ This turns the two original hyperplanes into $\alpha_{v_A}=d- \sum_{v\in A'} \alpha_v = d- \sum_{v\in A'} i_v$ and $\alpha_{v_B}=d- \sum_{v\in B'} \alpha_v = d- \sum_{v\in B'} i_v$ where $\{v_A\}=A \setminus A'$ and $\{v_B\}=B \setminus B',$ which shows that the point is in fact a lattice point.
\end{proof}

\section{Four polyhedral complexes to build cellular resolutions}

\begin{definition} \label{def:complexes}
We define three $(n-2)$-dimensional polyhedral complexes embedded in $\mathbb{R}^{n}.$ They are all subdivisions of $\textrm{Newt}(I_{P_n}^d).$
\begin{itemize}
\item[1)] For integers $0\leq i \leq n$ and $0 \leq j,$ define the hyperplanes
\[
H_{i,j} = \left\{ \mathbf{y} \in \mathbb{R}^n \left|  y_i  = j   \right.  \right\}
\]
and subdivide  $\textrm{Newt}(I_{P_n}^d)$ by all $H_{i,j}$ to get $W_n^d.$
\item[2)] For integers $0\leq i \leq n$ and $0 \leq j,$ define the hyperplanes
\[
H'_{i,j}=\left\{ \mathbf{y} \in \mathbb{R}^n \left|    \sum_{k=0}^{  \lfloor (i-1)/2 \rfloor }  y_{i-2k} = j   \right.  \right\}
\]
and subdivide  $\textrm{Newt}(I_{P_n}^d)$ by all $H'_{i,j}$ to get $Y_n^d.$
\item[3)] The common refinement of  $\textrm{Newt}(I_{P_n}^d)$ by subdividing by both $H_{i,j}$ and $H'_{i,j}$ is $Z^d_n.$
\end{itemize}
\end{definition}

The following polyhedral complex was defined by Dochtermann and Engstr\"om in Definition 3.1 and Section 5 of \cite{dochtermannEngstrom2012} and was employed to find cellular resolutions of cointerval ideals.
\begin{definition}\label{def:homComplex}
For positive integers $n$ and $d$, start with the $(n-d)$-simplex in $\mathbb{R}^{n-d+1}$ spanned by $d\textrm{\bf e}_1,\ldots,d\textrm{\bf e}_{n-d+1}$ and then subdivide the simplex by the hyperplanes defined by\linebreak $\mathbf{y}\cdot (\mathbf{e}_1 + \cdots + \mathbf{e}_i) = j $ for all integers $i$ and $j.$  This is a geometric realization of  $X_{d,n},$ a subcomplex of $\prod_{i=1}^d \Delta_n,$ where $\Delta_n$ is a simplex with vertex set $1,2,\ldots, n.$ It is the induced subcomplex on vertices $(a_1,\ldots, a_d)$ satisfying $a_1<a_2<\ldots <a_d.$ In this geometric realization, the vertex $(a_1,\ldots, a_d)$ is realized as $\sum_{i=1}^d \mathbf{e}_{a_i-i+1}.$
\end{definition}

\begin{proposition}\label{prop:realiz}
The polyhedral complex $Y_n^d$ is a geometric realization of $X_{d,d+n-2}$.
\end{proposition}
\begin{proof}
The linear transformation defined by sending $\mathbf{e}_i$ to  $\mathbf{e}_i+\mathbf{e}_{i+1}$ for all $i$ sends $X_{d,d+n-2}$ to $Y_n^d,$ since the vertices of the underlying simplices are sent to each other and the hyperplane $y_1+\cdots+y_i=j$ is sent to $ \sum_{k=0}^{  \lfloor (i-1)/2 \rfloor }  y_{i-2k} = j .$
\end{proof}

By the linear map, all vertices of  $Y_n^d$ are lattice points.  The vertex realized as
\[ (\mathbf{e}_{a_1} + \mathbf{e}_{a_1+1} )+  (\mathbf{e}_{a_2} + \mathbf{e}_{a_2+1} ) + \cdots + ( \mathbf{e}_{a_d} + \mathbf{e}_{a_d+1}) \]
where $a_1 \leq a_2 \leq \cdots \leq a_d$ is $(a_1,a_2+1,\ldots, a_d+d-1)$ in the definition of  $X_{d,d+n-2}$ as a polyhedral complex.

\section{The complex $Z^d_n$ supports a cellular resolution of $I_{P_n}^d$}

\begin{definition}
Let $X$ be a cell complex with monomial labeling $\ell$ and let $\alpha$ be a monomial. The complex $X_{\le \alpha}$ is the subcomplex of $X$ consisting of all cells $\sigma$ for which $\ell(\sigma)$ divides $\alpha$.
\end{definition}

The following theorem establishes which labeled complexes support resolutions. We are not in the restricted setting of \cite{MS} where all cell complexes are polyhedral, so we need the generality of \cite{BW} in which all details can be found.

\begin{theorem}
Let $X$ be a cell complex with monomial labeling $\ell$. If $X_{\leq \alpha}$ is acyclic or empty for all $\alpha$, then $X$ supports a cellular resolution of the ideal $\langle \ell(v)\mid \textrm{$v$ is a vertex of $X$}\rangle$. Moreover, if no cells properly included in each other have the same label, then the resolution is minimal.
\end{theorem}

Newton polytopes can be used to construct cellular resolutions.

\begin{proposition}\label{prop:base}
Let $I$ be a monomial ideal in $\mathbf{k}[x_1,\ldots,x_n]$ whose Newton polytope $P$ is normal and,
for some positive integer $d$,
let $X$ be the subdivision of $dP$ by integer translations of the coordinate hyperplanes. 
If the vertices of $X$ are the lattice points of $dP,$ then it supports a cellular resolution of $I^d$
with the monomial labeling given by the coordinates.
\end{proposition}

\begin{proof}
The monomials given by the lattice points of $dP$ generate $I^d$ since $P$ is normal. If $X$ supports a cellular resolution, then it does for $I^d.$  We now show that $X_{\leq \alpha}$ is acyclic or empty for all $\alpha \in \mathbb{Z}_{\geq 0}^n.$ The subcomplex $X_{\leq \alpha}$ is 
$dP\cap\{\mathbf{y}\in\mathbb{R}^n\mid \mathbf{y}\le \alpha  \},$
since a cell is in $X_{\leq \alpha}$ if all of its vertices are contained in $\{\mathbf{y}\in\mathbb{R}^n\mid \mathbf{y}\le \alpha \}$. If $X_{\leq \alpha}$ is non-empty, then it is convex and acyclic.
\end{proof}

\begin{proposition}\label{prop:prev}
Let $G$ be a bipartite graph. Then subdividing $d \cdot \textrm{\emph{Newt}}(I_G)$ by all integer translates of the coordinate hyperplanes, and labeling by coordinates, gives a complex supporting a cellular resolution of $I_G^d.$
\end{proposition}

\begin{proof}
This follows directly from Propositions~\ref{prop:semiHibi} and~\ref{prop:base}.
\end{proof}
\begin{remark}
It follows from Theorem 5.9 in \cite{SVV}, as explained in Lemma 2.6 of \cite{M}, that the dimension $n-2$ is minimal for high powers.
\end{remark}

Note, in particular, this shows that $W^d_n$ supports a cellular resolution of $I_{P_n}^d.$

\begin{definition}
A zero/one--matrix  $A$ has the \emph{consecutive-ones} property if it can be permuted into a matrix such that on each row the ones occur consecutively.
\end{definition}

\begin{proposition}\label{prop:zero}
If a zero/one--matrix $A$ has the consecutive-ones property and is invertible, then $A^{-1}$ is an integer matrix.
\end{proposition}
\begin{proof}
An integer matrix is \emph{totally unimodular} if the determinant of every square sub-matrix of it is -1, 0, or 1. For the basic theory of these matrices, see \cite{schrijver86}. A zero/one--matrix with the ones consecutive in each row is totally unimodular. A unimodular matrix that is invertible has an integer inverse matrix. 
\end{proof}

\begin{lemma}\label{lem:vertexLabels}
The vertices of  $W^d_n, Y^d_n,$ and $Z_n^d$ are labelled by the generators of $I_{P_n}^d.$
\end{lemma}
\begin{proof}
By Proposition~\ref{prop:prev}, the complex $W^d_n$ supports a cellular resolution o $I^d_{P_n}$ and so the vertices are labeled by the generators of  $I_{P_n}^d.$

By the linear map in Proposition~\ref{prop:realiz}, the vertices of $Y^d_n$ are lattice points. By the discussion after Proposition~\ref{prop:realiz}, all lattice points in $\textrm{{Newt}}(I_G)$ are vertices since there is a vertex in $X_{d,d+n-2}$ for every lattice point in $\textrm{{Newt}}(I_G)$.

It remains is to show that $Z^d_n$ only has lattice points as vertices. A vertex is defined as the unique solution to a system of equations $Ay=b$ where each row is the equation of one of the defining hyperplanes of the polyhedral complex.
The defining hyperplanes of the complex $Z^d_n$ are of the form $H_{i,j}$ or $H'_{i,j}$.
In the equations for these hyperplanes in Definition~\ref{def:complexes}, there is either only odd coordinates $y_{2i+1}$ occurring or only even coordinates $y_{2i+1}$ occuring. Permuting the columns of $A$ so that the first columns are the odd columns $y_1,y_3,\ldots$ and then followed by the even columns $y_2,y_4,\ldots$ shows that $A$ has the consecutive-ones property.
The matrix $A$ is invertible, as it defines a vertex, and by Proposition~\ref{prop:zero}, the matrix $A$ has an integer inverse. The column vector $b$ is integer and thus, $y$ is a lattice point.
\end{proof}

\begin{proposition}\label{prop:cellZ}
The complex $Z_n^d$ supports a cellular resolution of $I^d_{P_n}.$
\end{proposition}
\begin{proof}
By Proposition \ref{prop:prev}, the complex $W^d_n$ supports a cellular resolution of $I^d_{P_n}.$ As $Z_n^d$ is a refinement of $W^d_n$ with no new vertices, according to Lemma \ref{lem:vertexLabels}, and the labels are from  coordinates, the subcomplexes $(Z_n^d)_{\le \alpha}$ are convex or empty.
\end{proof}

\begin{remark}
The argument given in Proposition \ref{prop:cellZ} would not work for $Y^d_n,$ because every $(Y^d_n)_{\leq \alpha}$
is not convex, although every cell of $Y^d_n$ is convex.
\end{remark}

In Figure~\ref{fig:medCordHyp1}, the complex $Z^2_5$ is depicted. Removing the red hyperplane $H_{3,1}$ gives the complex $Y^2_5$. The two toblerone cells in $Y^2_5$ are subdivided by the hyperplane $H_{3,1}$. The Morse matching used to remove the hyperplane $H_{3,1}$ only matches cells in $(Z^2_5)_{\le 12111}$ and $(Z^2_5)_{\le 11121}.$ These complexes are depicted in green in Figures~\ref{fig:medCordHyp2} and~\ref{fig:medCordHyp3}. The green tetrahedron is matched with the green triangle not on the boundary of the toblerone cell. The remaining green triangles that are not in $Y^2_5$ are matched to the edges that are not in $Y^2_5$.

 \begin{figure}
 \begin{center}
  \includegraphics[width=70.0mm]{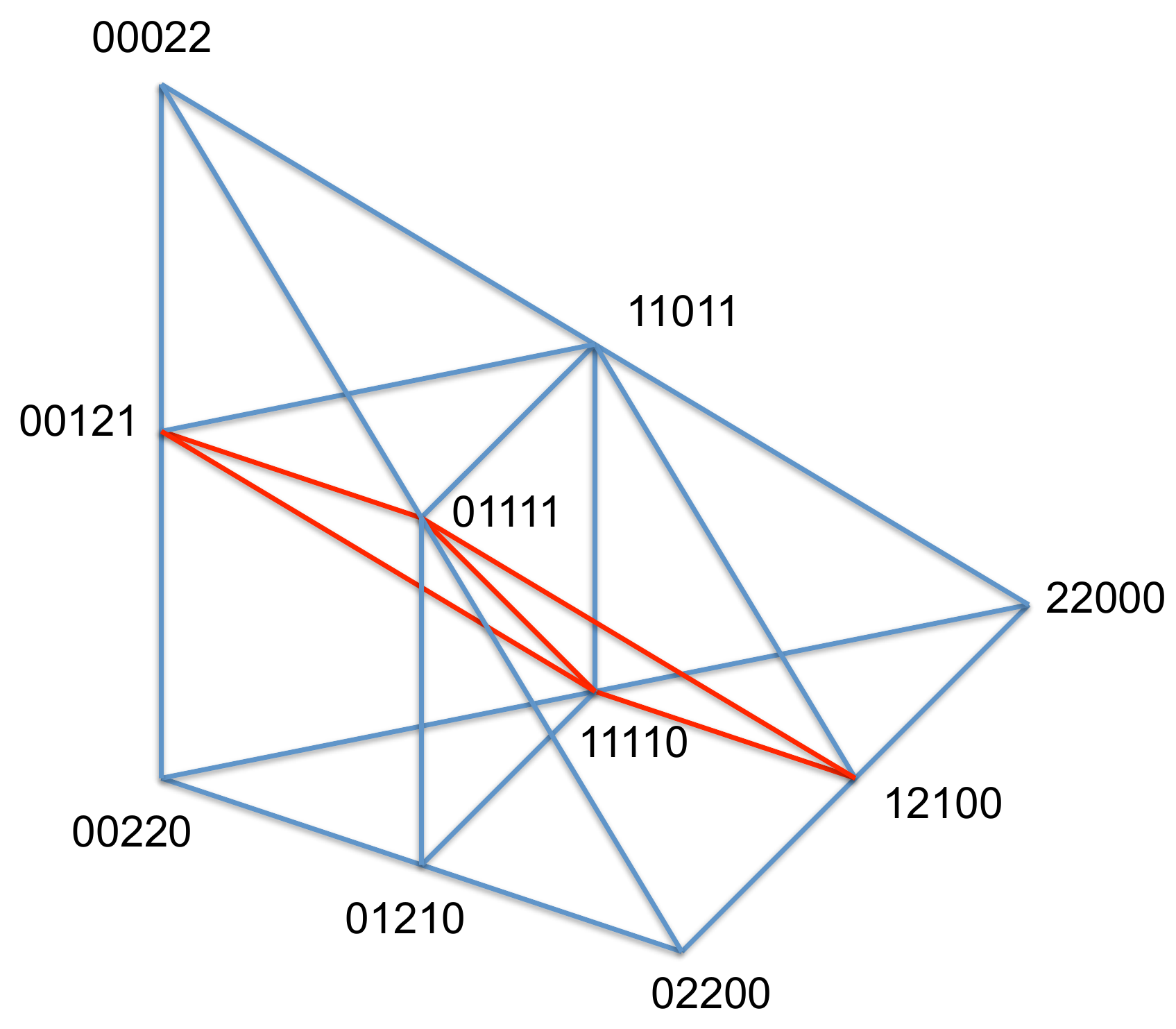}
 \caption{ The complex $Z^2_5$. The red hyperplane is the coordinate hyperplane $H_{3,1}$ that will be removed.}\label{fig:medCordHyp1}
 \end{center} 
 \end{figure}

\begin{figure}
 \begin{center}
  \includegraphics[width=140.0mm]{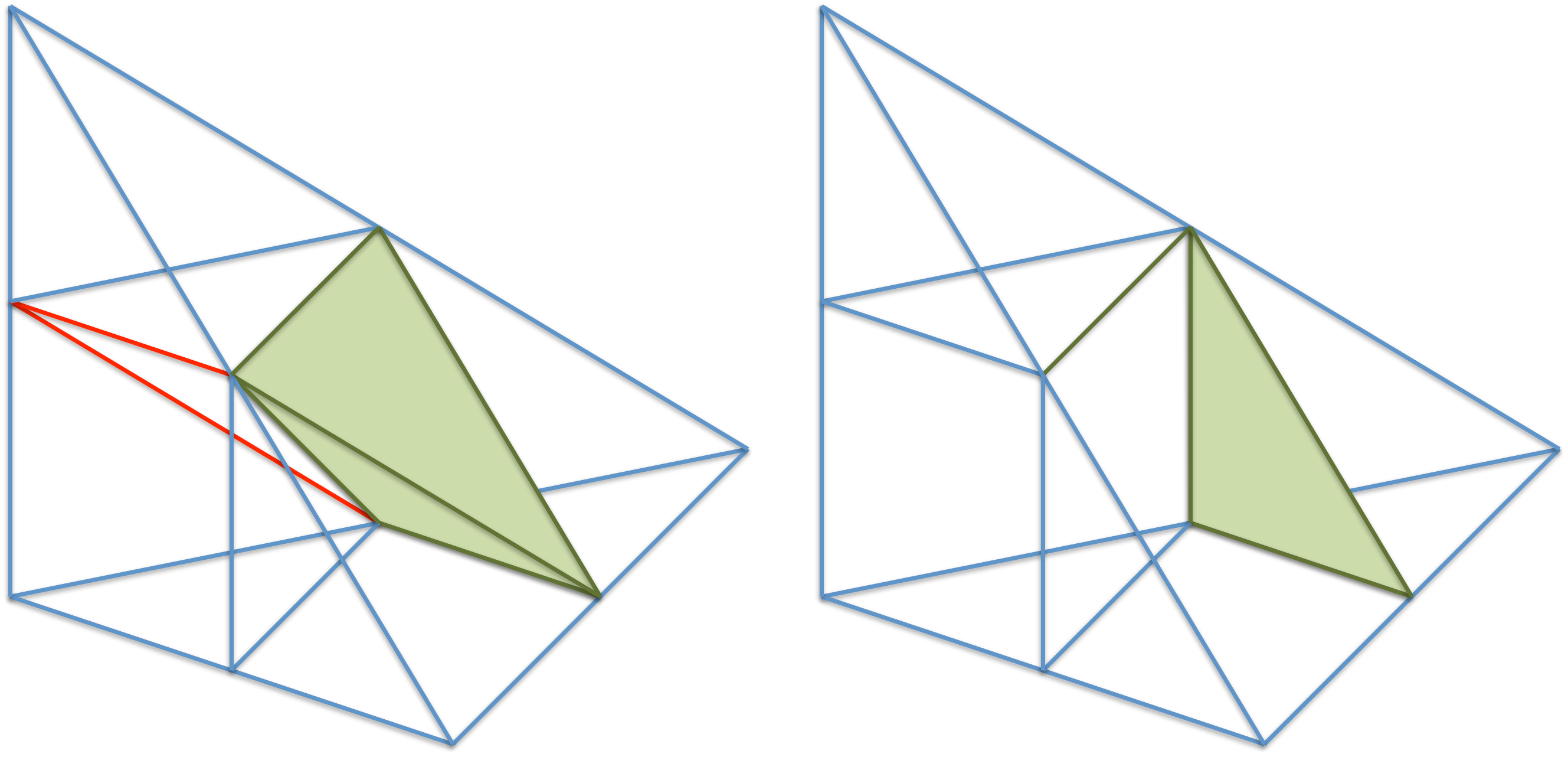}
 \caption{The subcomplex $\leq 12111$ in green. The left picture is the subocmplex in $Z^2_5$. The right picture is in $Y^d_5$ with  the red hyperplane $H_{3,1}$ removed.}
 \label{fig:medCordHyp2}
 \end{center} 
 \end{figure}

\begin{figure}
 \begin{center}
  \includegraphics[width=140.0mm]{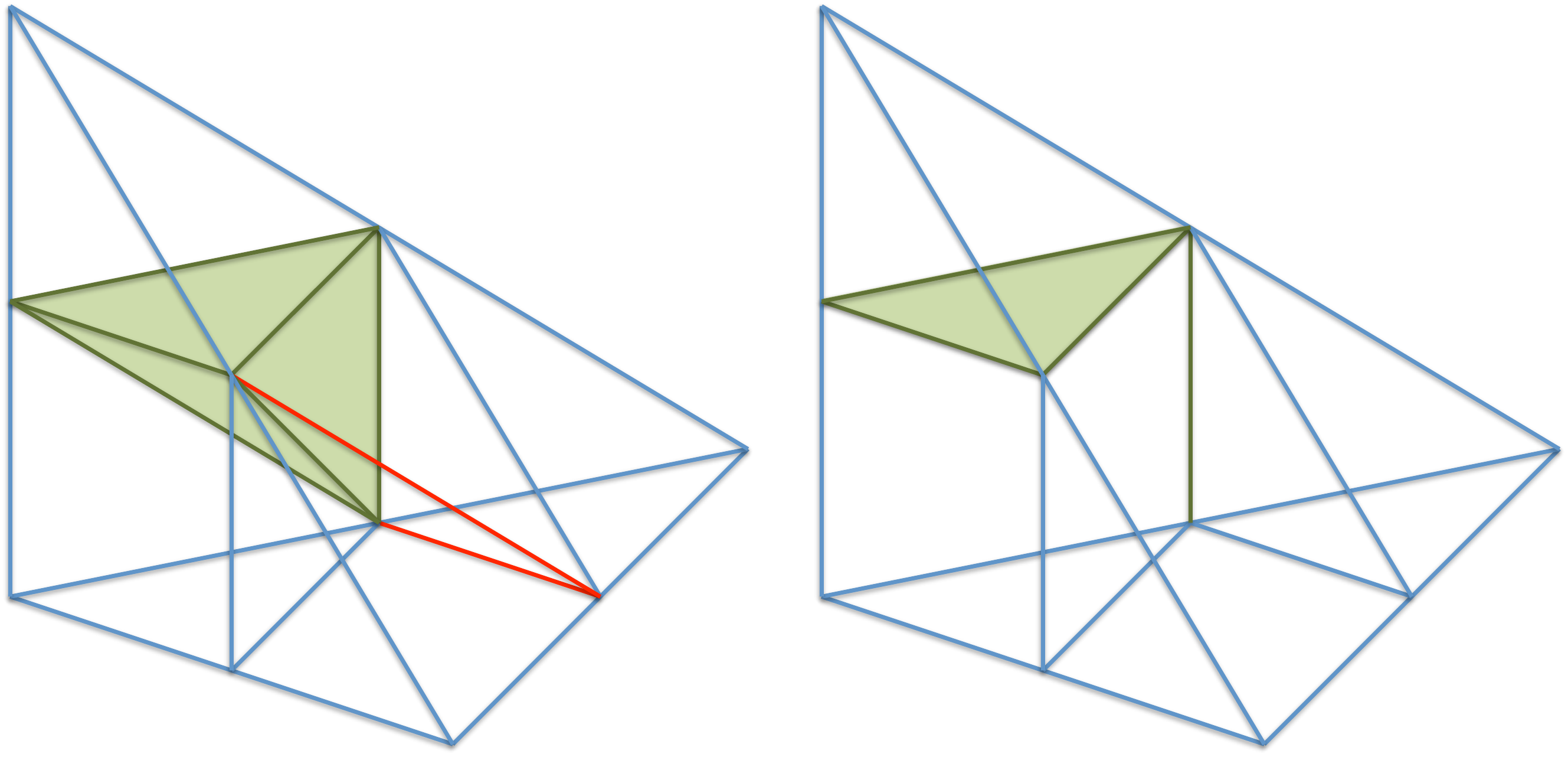}
 \caption{The subcomplex $\leq 11121$ in green. The left picture is the subcomplex in $Z^2_5$. The right picture is in $Y^d_5$ with  the red hyperplane $H_{3,1}$ removed.}
  \label{fig:medCordHyp3}
 \end{center} 
 \end{figure}

\section{The complex $Y^d_n$ supports a cellular resolution of $I_{P_n}^d$}
The resolutions obtained in the previous section are often not minimal. Using discrete Morse theory, it is possible to make the resolutions smaller and sometimes minimal.

\begin{theorem}[The main theorem of discrete Morse theory]
Let $X$ be a regular CW-complex with face poset $P.$ If $M$ is an acyclic matching on $P,$ then there is a CW-complex $\tilde{X}$ whose cells correspond to the critical cells and they are homotopy equivalent.
\end{theorem}

It is possible to extend this to work for cellular resolutions. This was done by Batzies and Welker~\cite{BW}.

\begin{theorem}\label{thm:alebraicmorsetheory}
Let $X$ be a cell complex supporting a cellular resolution and $M$ a Morse matching of this complex. If $M$ only matches cells with the same labels, then the Morse complex $\tilde{X}$ also supports a cellular resolution of the same ideal.
\end{theorem}

The complexes $X_{d,d+n-2}$ and $Y^d_n$ are isomorphic. The embedding of $Y^d_n$ is useful and so is the description of $X_{d,d+n-2}$ as a product of simplices. The diagrams in Figure~\ref{fig:staircases} illustrate this. A cell $\sigma=\sigma_1\times\cdots\times\sigma_d$ is represented by the diagram containing the boxes with labels in $\sigma_k$ on row $k$. Vertices of $Y^d_n$ correspond to exactly one box on each row. The diagrams of maximal cells form staircases.
 \begin{figure}
 \begin{center}
  \includegraphics[width=\textwidth]{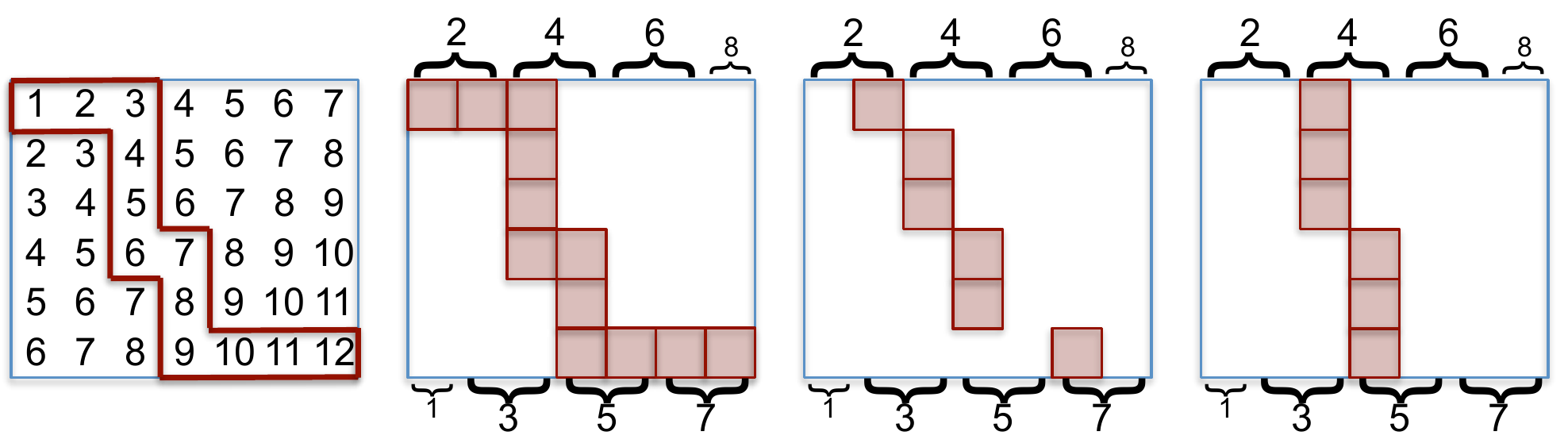}
 \caption{The cell $\{1,2,3 \}\times\{4 \}\times\{ 5\}\times\{6,7 \}\times\{8 \}\times\{9,10,11,12 \}$ of $Y^6_8$ in staircase notation with its vertices
 $\{2 \}\times\{ 4\}\times\{ 5\}\times\{ 7\}\times\{8 \}\times\{ 11 \}$ and $\{3 \}\times\{4 \}\times\{ 5\}\times\{7 \}\times\{8 \}\times\{ 9\}.$
 The geometric realizations of the vertices are given by counting the number of boxes:
 $(0,1,3,4,2,1,1,0)$ and $(0,0,3,6,3,0,0,0)$. The monomial labels also follows from counting the boxes. For the vertices, they are $x_2x_3^3x_4^4x_5^2x_6x_7$
 and $x_3^3x_4^6x_5^3,$ and the label of the cell is $x_1x_2^2x_3^5x_4^7x_5^4x_6^2x_7^2x_8.$}\label{fig:staircases}
 \end{center} 
 \end{figure}

To construct a cellular resolution supported by $Y^d_n,$ we find a Morse matching on $Z^d_n$ respecting the labels such that the
Morse complex is $Y^d_n.$ To describe the matchings, we first explain how the translated coordinate hyperplanes 
intersect the cells of $Y^d_n.$ As described in \cite{dochtermannEngstrom2012}, the Cayley trick gives a
connection between staircase triangulations of a product of two specific simplices and $X^d_n.$ The
Cayley trick and staircase triangulations are surveyed in \cite{deLoeraEtAlBook}, but we only need a piece
of the notation for cells of $Y^d_n$ that are explained in Figure~\ref{fig:staircases}.

\begin{definition}
Let $\sigma=\sigma_1\times\cdots\times\sigma_d$ be an open cell of $X_{d,d+n-2}$ with the embedding of $Y^d_n$. A hyperplane $H_{i,j} = \left\{ \mathbf{y} \in \mathbb{R}^n \left|  y_i  = j   \right.  \right\},$ as in Definition \ref{def:complexes}, is \emph{$\sigma$--subdividing} if it intersects $\sigma$. Define ${\mathrm sd}_i(\sigma)$ to be the set of open cells obtained from subdividing $\sigma$ with all $\sigma$--subdividing hyperplanes $H_{i',j}$ for $i'\ge i$.
\end{definition}

\begin{remark}
Both ${\mathrm sd}_i(\sigma)$ and  ${\mathrm sd}_{i+1}(\sigma)$ are sets of open cells from subdivisions of $\sigma.$ Note that there is a filtration: We get ${\mathrm sd}_i(\sigma)$ from ${\mathrm sd}_{i+1}(\sigma)$
by further subdividing by the system of parallel hyperplanes $H_{i,j} = \left\{ \mathbf{y} \in \mathbb{R}^n \left|  y_i  = j   \right.  \right\}$ for all integers $j.$ Another important property of the open cells in these subdivisions is that they are convex, since the cells we subdivide are already convex by definition.
\end{remark}

\begin{definition}
Let $\sigma=\sigma_1\times\cdots\times\sigma_d$ be a cell of $X_{d,d+n}$.
An element $j$ of $\sigma_i$ \emph{covers} the vertices $(j-i+1)$ and $(j-i+2)$ of $P_n.$ The set $\sigma_i$ \emph{covers} $k$, if a $j$ of $\sigma_i$ covers $k$. 
\end{definition}

At this, point we advise the reader to return to Figure~\ref{fig:staircases} and note that the numbers with curly brackets show the coverings. This notion is important in several technical proofs.

\begin{lemma}\label{lemma:match}
Let $\sigma=\sigma_1\times\cdots\times\sigma_d$ be an open cell of $X_{d,d+n-2}$ with the embedding of $Y^d_n$. Let $\tau$ be a $(d'-1)$--dimensional cell in ${\mathrm sd}_i(\sigma)$, but not in ${\mathrm sd}_{i+1}(\sigma)$. If $\tau$ is contained in the hyperplane $H_{i,j},$ then there is a unique $d'$-dimensional cell $\tau_-$ in 
${\mathrm sd}_i(\sigma)$ that both has $\tau$ on its boundary and all points $\mathbf{y}$ in $\tau_-$ satisfy that $y_i<j$. Furthermore, the labels of $\tau$ and $\tau_-$ are the same.
\end{lemma}
\begin{proof}
First we make use of the convexity of the open cells in the subdivisions. The $(d'-1)$--dimensional cell $\tau$ is contained in a $d'$-dimensional cell $\tau'$ in ${\mathrm sd}_{i+1}(\sigma)$. The parallel hyperplanes
\[\ldots, H_{i,j-2}, H_{i,j-1},  H_{i,j}, H_{i,j+1}, H_{i,j+2},.. \]  
slice the convex open cell $\tau'$ into pieces ending up in ${\mathrm sd}_i(\sigma).$ The cell between $H_{i,j-1}$ and  $H_{i,j}$ is denoted $\tau_-.$ The cell between $H_{i,j}$ and $H_{i,j+1}$ is denoted $\tau_+.$ The cell $\tau_-$ has $\tau$ on its boundary, $y_i<j$ for all of its points, and it is clearly the unique cell with that property. 

Let $\ell$ be the monomial labeling from coordinates. It remains to show that $\ell(\tau)=\ell(\tau_-)$. 
The inequality $\ell_{i'}(\tau_-) \geq \ell_{i'}(\tau)$ follows for all $i'$ from that $\tau$ is on the boundary of $\tau_-.$ 
The proof of $\ell_{i'}(\tau_-) \leq \ell_{i'}(\tau)$ is shown for different $i'$ in four cases.
\begin{itemize}
\item[I.] The case $i'=i.$

\noindent
The maximal $y_i$ on the boundary of both $\tau$ and $\tau_-$ is $j,$ so $\ell_i(\tau)=\ell_i(\tau_-)=j.$

\item[II.] The case $i'>i.$

\noindent
{\textbf{Claim.}} The closure of $\tau'$ does not intersect the hyperplane $H_{i',\ell_{i'}(\tau)+1}.$

The inequality $\ell_{i'}(\tau_-)\leq \ell_{i'}(\tau')$ follows from $\tau_- \subset \tau'$
and $\ell_{i'}(\tau') \leq \ell_{i'}(\tau)$ follows from the claim, and that the closures of $\tau \subset \tau'$ intersect the hyperplane $H_{i',\ell_{i'}(\tau)}.$ Thus, $\ell_i(\tau_-) \leq \ell_i(\tau).$

\noindent
{\emph{Proof of claim.}} Assume the contrary. As $i'>i,$ either both of $\tau$ and $\tau'$ are contained in the hyperplane $H_{i',\ell_{i'}(\tau)}$ or neither of them are. By assumption, $\tau'$ is not contained in 
$H_{i',\ell_{i'}(\tau)}$ since it intersects the parallel hyperplane $H_{i',\ell_{i'}(\tau)+1}$ and so neither is $\tau$. Since $\tau$ is in a subdivision generated in parts by intersecting with $H_{i',\ell_{i'}(\tau)},$ but is not contained in it, $\tau$ is on one side of $H_{i',\ell_{i'}(\tau)}.$ That is, $y_{i'}<\ell_{i'}(\tau)$ or $y_{i'}>\ell_{i'}(\tau)$ for all $y$ in $\tau.$ The second option is never true for any labeling, so $y_{i'}<\ell_{i'}(\tau)$ for all $y$ in $\tau.$ When subdividing by intersecting with $H_{i',\ast}$ hyperplanes, the new cells in the refined subdivision either end up between or in these hyperplanes. In particular, no cell can have points in its closure on different sides of an hyperplane. However, $y_{i'}<\ell_{i'}(\tau)$ for all $y$ in $\tau\subset \tau'$ and the closure of $\tau'$ intersects the hyperplane $H_{i',\ell_{i'}(\tau)+1}$ by assumption. Thus, $\tau'$ contains points in its closure on different sides of the hyperplane $H_{i',\ell_{i'}(\tau)},$ a contradiction, and hence the claim is proved.

\item[III.] The case $i'<i-1.$

Let $v^-=v^-_1 \times\cdots\times v^-_d$ be a vertex  of $\tau_-$ with $\ell_{i'}(v^-)=\ell_{i'}(\tau_-)$. If $\ell_i(v^-)=\ell_i(\tau)=j,$ then $v^-$ is a vertex of $\tau$ and $\ell_{i'}(\tau)\ge\ell_{i'}(\tau_-)$. 

The remaining subcase $\ell_i(v^-)=j-1$. By definition, the cell $\tau_+$ has some vertex $v^+$ on its boundary with $\ell_i(v^+)=j+1.$
Let $k$ and $r$ be the smallest and largest elements of $\{ l \mid \textrm{$\sigma_{l}$ covers $i$} \}.$
From the existence of vertices in $\sigma$ with monomial labels of $x_i$--degree $j-1$, $j$, and $j+1$, it follows that $r-k=j>0$ and that both $\sigma_k$ and $\sigma_r$ contains elements not covering $i.$

The cell $\tau$ is by definition not on the boundary of $\sigma=\sigma_1\times\cdots\times\sigma_d$. For every $v_s\in \sigma_s$, there is a vertex
$v=v_1\times\cdots\times v_d$ in the closure of $\tau$. Otherwise, $\tau$ would be on the boundary cell $\sigma_1\times\cdots\times\sigma_{s-1}\times(\sigma_s\setminus \{v_s\})\times\sigma_{s+1}\times\cdots\times\sigma_d$ of $\sigma$. In particular, we can choose a vertex $v=v_1\times\cdots\times v_d$ of $\tau$ with $v_r=\min\sigma_r$

Now, we show that $v_1,\ldots,v_k$ and $v^-_1,\ldots,v^-_k$ only cover elements smaller than $i$. By construction, it is enough to show that neither $v_k$, nor $v^-_k$, covers $i$. All elements of the $j-1$ cells $\sigma_{k+1},\ldots,\sigma_{r-1}$ cover $i.$ The label $\ell_{i}(v^-)=j-1$ shows that no other elements of $v^-$ cover $i$ and, in particular, $v^-_k$ does not cover $i.$ For $v$ the situation is similar. Its corresponding label is $\ell_{i}(v)=j,$ but $v_r$ covers $i$ by construction. This shows that $v_k$ does not cover $i.$

Since both $v_{k+1}$ and $v_{k+1}^-$ cover $i,$ this shows that $v_{k+1},\ldots,v_d$ and  $v_{k+1}^-,\ldots,v_d^-$ only covers elements larger or equal to $i-1.$
Define a vertex $w=v^-_1\times\cdots\times v^-_k\times v_{k+1}\times \cdots \times v_d$ of $\sigma$ to verify Case III.
It is a vertex of $\tau$ since $\ell_i(w)=\ell_i(v)=j$ and $ \ell_{i'}(w)=\ell_{i'}(v^-)=\ell_{i'}(\tau_-)$ since $i'<i-1.$

\item[IV.] The case $i'=i-1.$

This case is split into seven different subcases depending on the structure of $\sigma=\sigma_1\times\cdots\times\sigma_d$. The definition of $k$ and $r$ are the same as in Case III and $k<r$ by the same argument. 
In Subcases 1-5, there is no $\sigma_s$ satisfying $k<s<r$ and $|\sigma_s|=2,$ and the two remaining subcases are 6-7. All of the subcases are drawn in Figure~\ref{fig:sevenCases}.

For subcases 1-6, first choose a vertex $v=v_1\times\cdots\times v_d$ of $\tau$ that uses the crossed box and a vertex $v^-=v^-_1\times\cdots\times v^-_d$ of $\tau_-$ with maximal $x_{i-1}$--degree. Depending on the subcase, set
$
(1) \, t=k, \,\,
(2) \, t=r, \,\,
(3) \, t=r, \,\,
(4) \, t=k, \,\,
(5) \, t=k+1, \,\,
(6) \, t=s, \,\,
$
and define $w=v^-_1\times\cdots\times v^-_{t-1}\times v_{t}\times\cdots\times v_d$ to verify Case IV in a similar manner as in Case III.

Finally, for Subcase 7, if there is a vertex of $\tau$ using both the boxes marked by a cross and a circle, then we are done. Otherwise, choose $v^-$ as above, $v$ as a vertex of $\tau$ containing the box with a circle, and use the vertex $w$ defined by $t=s$ to verify Case IV.
\end{itemize}
\end{proof}

 \begin{figure}
 \begin{center}
  \includegraphics[width=100mm]{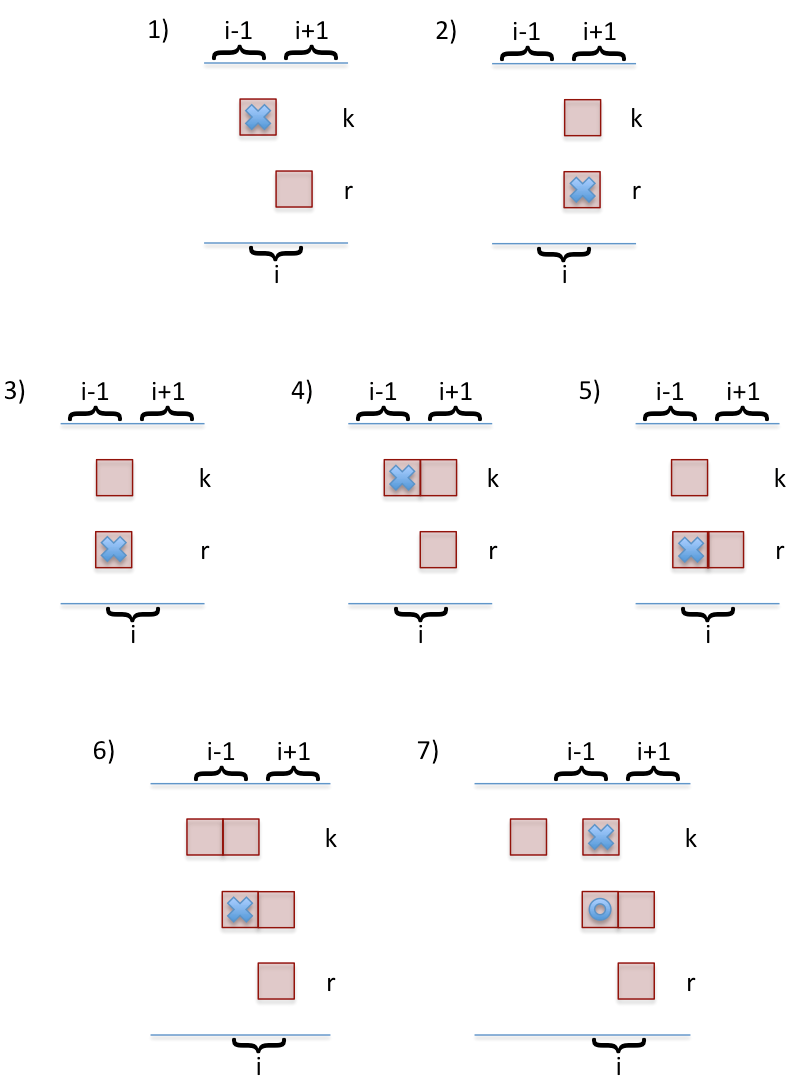}
 \caption{The seven subcases of case IV of Lemma~\ref{lemma:match}.}\label{fig:sevenCases}
 \end{center} 
 \end{figure}

\begin{theorem}\label{thm:Ysupports}
The embedded and labeled complex $Y^d_n$ supports a cellular resolution of $I^d_{P_n}$.
\end{theorem}
\begin{proof}
The matching implied by Lemma~\ref{lemma:match} together with Theorem~\ref{thm:alebraicmorsetheory} and Proposition~\ref{prop:cellZ} establishes this.
\end{proof}

\section{A minimal cellular resolution of $I_{P_n}^d.$}

To make cellular resolutions minimal, the resolutions supporting complexes are partitioned into pieces and discrete Morse theory is employed on each piece to reduce the size to the minimal one. The face poset of each piece is essentially the Alexander dual of  the independence complex of a graph. On a poset level, this is not a new simplicial complex \cite{bjornerButlerMatveev97} and for independence complexes of ordinary graphs this was made explicit in \cite{kawamura11}. Even though there is a connection on the level of (co)homology, there is no straight forward duality theory for discrete Morse theory \cite{benedettiXX} moving critical cells from the complex to its dual. Guided by results for independence complexes, as in \cite{engstrom09a}, we will study their `dual', the \emph{covering complex}, and find optimal discrete Morse matchings.

\begin{definition} \label{def:coveringComplex}
Let $G$ be a graph. The \emph{independence complex} ${\tt Ind}(G)$ is an abstract simplicial complex whose vertex set is $V(G)$ and $\sigma \in {\tt Ind}(G)$ if for every $e\in E(G)$, there is a $v \in e \setminus \sigma.$ The  \emph{covering complex} ${\tt Cov}(G)$ is an abstract simplicial complex whose vertex set is $E(G)$ and $\sigma \in {\tt Cov}(G)$ if for every $v\in V(G)$, there is an $e\in E(G) \setminus \sigma$ such that $v\in e.$
\end{definition}

Informally, the faces of a covering complex consists of all collections of edges of a graph such that the remaining edges covers the vertices of the graph.

\begin{proposition}\label{prop:coveringPath}
There is an acyclic matching on ${\tt Cov}(P_n)$ with
\begin{itemize}
\item[(i)] one critical cell on $(n-3)/3$ vertices if $n \equiv 0$ mod $3$;
\item[(ii)] no critical cells if $n \equiv 1$ mod $3$;
\item[(iii)] one critical cell on $(n-2)/3$ vertices if $n \equiv 2$ mod $3$.
\end{itemize}
\end{proposition}
\begin{proof}
The vertices corresponding to the edges $12$ and $(n-1)n$ are never in ${\tt Cov}(P_n).$ For the remaining edges $23,34,\ldots,(n-2)(n-1)$ there is a bijection between
${\tt Ind}(P_{n-3})$ and ${\tt Cov}(P_n)$ given by extending the vertex bijection $i \mapsto (i+1)(i+2).$ The optimal acyclic matching giving those critical cells for the independence complex is constructed in 
\cite{engstrom09a}.
\end{proof}

\begin{remark}
The critical cells are given by taking every third vertex/edge.
\end{remark}

\begin{proposition}\label{prop:coveringDisjointPath}
Let $G=\sqcup_{i=1}^t P_{n_i}$ be a disjoint union of paths. Then there is an acyclic matching on ${\tt Cov}(G)$ with at most one critical cell.
\end{proposition}
\begin{proof}
For simplicial complexes $\Sigma_1, \Sigma_2, \ldots, \Sigma_t$ with acyclic matchings on $c_1, c_2, \ldots, c_t$ critical cells, there is an acyclic matching on 
the join $\Sigma_1 \ast  \Sigma_2 \ast \cdots \ast \Sigma_t$ with $c_1c_2\cdots c_n$ critical cells. It follows from the definition that
 ${\tt Cov}(G)={\tt Cov}(P_{n_1})\ast {\tt Cov}(P_{n_2}) \ast \cdots \ast {\tt Cov}(P_{n_t})$ and from Proposition \ref{prop:coveringPath} that there is at most one critical cell for each ${\tt Cov}(P_{n_i}).$ This gives the desired acyclic matching.
\end{proof}

In order to describe an optimal algebraic discrete morse matching of $Y^d_n$, it is useful to express the labels with some new notation. Consider the cell $\sigma=\sigma_1 \times \cdots \times \sigma_6 =\{1,2,3 \}\times\{4 \}\times\{ 5\}\times\{6,7 \}\times\{8 \}\times\{9,10,11,12 \}$ depicted in Figure~\ref{fig:staircases}. Each $\sigma_i$ covers some vertices of the path $P_8,$ for example $\{1,2,3 \}$ covers $\{1,2,3,4\},$ and $\{6,7 \}$ covers $\{3,4,5\}$ according to Figure~\ref{fig:staircases}. We formalize this in a definition.

\begin{definition}
Let $\sigma=\sigma_1\times\cdots\times\sigma_d$ be a cell in $Y^d_n$. Then $V(\sigma_i) = \cup_{j \in \sigma_i} \{ j-i+1,j-i+2 \}.$
\end{definition}

This is a convenient and straight-forward lemma whose proof we omit.

\begin{lemma}\label{lemma:triv}
Let $s_1, \ldots s_d$ be non-empty sets of monomials, then
$\textrm{\emph{lcm}}(m_1\cdots m_d\mid m_i\in s_i)=\textrm{\emph{lcm}}(s_1)\cdots\textrm{\emph{lcm}}(s_d)$
\end{lemma}

\begin{proposition}\label{prop:label}
The label of the cell $\sigma=\sigma_1\times\cdots\times\sigma_d$ in $Y^d_n$ is $\prod_{i=1}^d\prod_{k\in V(\sigma_i)}x_k$.
\end{proposition}
\begin{proof}
The label of a cell is the least common multiple of the labels of its vertices. By Lemma~\ref{lemma:triv} and the geometric realization of $Y^d_n$ with monomial labels given by the coordinates,
\[
\begin{array}{rcl}
\ell(\sigma) & = &  \textrm{{lcm}} ( \ell( \{j_1\} \times \cdots \times  \{j_d\}) \mid  \{j_1\} \times \cdots \times  \{j_d\}  \subseteq  \sigma_1\times\cdots\times\sigma_d  ) \\  
& = & \textrm{{lcm}} ( \ell( \{j_1\})  \cdots \ell(  \{j_d\}) \mid  j_1\in \sigma_1, \ldots, j_d \in \sigma_d  ) \\  
& = & \prod_{i=1}^d  \textrm{{lcm}} ( \ell( \{j\})   \mid  j \in \sigma_i  ) \\  
& = & \prod_{i=1}^d  \textrm{{lcm}} (  x_{j-i+1}x_{j-i+2}   \mid  j \in \sigma_i  ) \\  
& = & \prod_{i=1}^d  \prod_{k \in \cup_{j \in \sigma_i} \{ j-i+1,j-i+2 \}} x_k   \\  
& = & \prod_{i=1}^d  \prod_{k \in V(\sigma_i)} x_k.  \\  
\end{array}
\]
\end{proof}

\begin{proposition}\label{prop:labelOnYdn}
The coordinate monomial labeling $\ell$ is a poset map from the face poset of $Y^d_n$ to the monomials in $\mathbf{k}[x_1,\ldots,x_n]$ ordered by divisibility. Let $q$ be a monomial in $\mathbf{k}[x_1,\ldots,x_n]$. The fiber $\ell^{-1}(q)$ is a disjoint union of connected posets.  The poset dual of the connected posets is isomorphic to  products of  face posets of covering complexes of disjoint unions of paths. 
\end{proposition}
\begin{proof}
Monomial labelings are defined by the least common multiple of the labels of the vertices, turning them into a poset map.

Let $\sigma=\sigma_1\times\cdots\times \sigma_d \supseteq \tau=\tau_1\times\cdots\times \tau_d$ be two comparable cells in the same fiber. By Proposition~\ref{prop:label},
\[ \prod_{i=1}^d\prod_{k\in V(\sigma_i)}x_k = \prod_{i=1}^d\prod_{k\in V(\tau_i)}x_k. \]
The inclusions $V(\sigma_i) \supseteq V(\tau_i)$ for all $i$ follows from $\sigma \supseteq \tau.$ Together this gives that $V(\sigma_i) = V(\tau_i)$ for all $i.$ Thus, in a connected component of the fiber, not only the label is common, but also $V(\sigma_i)$ for all $\sigma.$

Fix a connected component of the fiber and set $V_i=V(\sigma_i)$ for all $\sigma$ in it. Ordering by inclusion, there is a maximal $\sigma_i,$ denoted by $\tilde{\sigma}_i,$ such that $V(\sigma_i)=V_i.$ 
Let $\pi$ be the map from $E(P_n)$ to $V(P_n)$ that sends $j(j+1)$ to $j$ and extend the map $\pi$ to the domain of subsets of $E(P_n).$
Then
\[
\begin{array}{rcl}
\{ \sigma_i  \mid V_i = V(\sigma_i) \} & = & \left\{ \sigma_i  \mid  V_i =  \cup_{j \in \sigma_i} \{ j-i+1,j-i+2 \}  \right\}  \\
 & = & \left\{ \tilde{\sigma}_i \setminus  \pi(\xi)  \mid  \xi \in  {\tt Cov}(P_n[V_i])  \right\}  \\
\end{array}
\]
and $\{ \sigma_i  \mid V_i = V(\sigma_i) \}$ is isomorphic to the dual of the face poset of the covering complex of $P_n[V_i],$ a disjoint union of paths.
\end{proof}

\section{Minimal cellular resolutions}

\begin{lemma}\label{lemma:boundary}
Let $P$ be the face poset of a regular CW-complex $X,$ $Q$ a poset, $\phi:P\rightarrow Q$ a poset map, 
\[ \phi^{-1}(q)= \bigsqcup_{i=1}^{n_q} P_{q,i}, \]
and $M_{q,i}$ an acyclic matching on $P_{q,i}$ with at most one critical cell for each $q\in Q$ and $1\leq i \leq n_q.$
Then $M=\cup_{q\in Q}\cup_{i=1}^{n_q} M_{q,i}$ is an acyclic matching on $P.$ 
By the main theorem of discrete Morse theory, $X$ has the same homology as a CW-complex $\tilde{X}$ whose cells are the critical cells of the acyclic matching $M$, but with new boundary maps. If $\sigma$ and $\tau$ are cells in $\tilde{X}$ corresponding to critical cells in the same fiber $\phi^{-1}(q)$ and $\partial$ is the boundary map on $\tilde{X},$ then $\sigma \cdot \partial \tau = 0.$
\end{lemma}
\begin{proof}
Lemma 4.2 in \cite{jonsson08} states that $M$ is an acyclic matching. The boundary maps in $\tilde{X}$ are calculated from gradient paths \cite{forman98}. A gradient path in $P$ is a list $\tau_1,\sigma_1,\tau_2,\sigma_2,\ldots,\tau_m,\sigma_m$ of cells of $X$ such that $\sigma_i$ is a codimension one
cell on the boundary of $\tau_i$ for all $i,$ and $\{\sigma_i,\tau_{i+1}\} \in M$ for $1 \leq i <m.$ If $\tau_1$ and
$\sigma_m$ are critical cells, then $\sigma_m \cdot \partial \tau_1 = 0$ in $\tilde{X}$ if there are no gradient paths from $\tau_1$ to $\sigma_m.$

Assume that there is a gradient path $\tau=\tau_1,\sigma_1,\tau_2,\sigma_2,\ldots,\tau_m,\sigma_m=\sigma$ in $P.$ In the poset $Q,$ $\phi(\tau_i)\geq \phi(\sigma_i)$ since $\phi$ is a poset map.
Matched cells are in the same fiber, providing equalities $\phi(\sigma_i)=\phi(\tau_{i+1}).$ It is stated that $\phi(\tau)=\phi(\sigma)=q$ and thus, all cells in the gradient path are in the fiber $\phi^{-1}(q).$ 
All of the gradient path is in some $P_{q,j}$ since $\phi^{-1}(q)$ is a disjoint union of posets, 
but $P_{q,j}$ only contains at most one critical cell. This contradicts the assumption that there is a gradient path.
\end{proof}

\begin{theorem}\label{thm:main}
There is a minimal cellular resolution of $I^d_{P_n}.$
\end{theorem}
\begin{proof}
According to Theorem~\ref{thm:Ysupports} the coordinate labeled complex $Y^d_n$ supports a cellular resolution of $I^d_{P_n}$. The structure of the cells in the face poset with a particular fixed monomial label is given by Proposition~\ref{prop:labelOnYdn}. They are a disjoint union of connected posets and the duals of the connected posets are isomorphic to  products of face posets of covering complexes of disjoint unions of paths. 
By Proposition~\ref{prop:coveringDisjointPath} there are acyclic matchings with at most one critical cell on face posets of covering complexes of disjoint unions of paths. Taking products of such posets and their duals also yields an acyclic matching with at most one critical cell.

We now construct acyclic matchings this way for each monomial label. Since the map giving labels is a poset map, the composed matching on all of the face poset of $Y^d_n$ is also an acyclic matching. Passing from $Y^d_n$ to its Morse complex $\tilde{Y}^d_n,$ that also supports a cellular resolution of $I^d_{P_n}$ by Theorem~\ref{thm:alebraicmorsetheory}, we get a cell complex whose labels drop when passing to the boundary of cells according to Lemma~\ref{lemma:boundary}. This shows that the resolution supported by $\tilde{Y}^d_n$ is minimal.
\end{proof}

A cell $\sigma$ is \emph{label maximal} if it is not contained in a cell with the same label. The face poset of $Y^d_n$ decomposes into disjoint parts consisting of cells with the same label, each part contains a unique label maximal cell. Let $\sigma$ be a label maximal cell, the poset of cells contained in $\sigma$ and having label $\ell(\sigma)$ is $F_{\sigma}$. A label maximal cell $\sigma$ is \emph{critical inducing} if $F_{\sigma}$ contains a critical cell. The goal is to count the number of label maximal critical inducing cells and keeping track of the dimension of the critical cell that are induced together with the label. The number $A(\sigma)=|E(\sqcup_{i=1}^dP_n[V(\sigma_i)])|$ will be important.

Let $\sigma$ be a critical inducing label maximal cell. The dimension of $\sigma$ is $|\sqcup_{i=1}^d \sigma_i|-d=|E(\sqcup_{i=1}^dP_n[V(\sigma_i)])|-d=A(\sigma)-d$. The poset $F_\sigma$ is dual to the product of the posets ${\tt Cov}(P_n[V(\sigma_i)])$, this product can be realised as the face poset of ${\tt Cov}(\sqcup_{i=1}^dP_n[V(\sigma_i)])$. The complex ${\tt Cov}(\sqcup_{i=1}^dP_n[V(\sigma_i)])$ has exactly one critical cell, let $D_\sigma$ be the dimension of the critical cell. Now the dimension of the critical cell in $F_\sigma$ is $A(\sigma)-(D_\sigma+1)$.

Let $N(\sigma)$ be the number of connected components of $\sqcup_{i=1}^dP_n[V(\sigma_i)]$ and let $N_2(\sigma)$ be the number of these components that have $2\mod 3$ vertices. The cell $\sigma$ is critical inducing and this implies that $\sqcup_{i=1}^dP_n[V(\sigma_i)]$ have no connected component with $1\mod 3$ vertices.

Proposition~\ref{prop:coveringPath} proves the formula $D_\sigma=(A(\sigma)+N_2(\sigma)-2N(\sigma)-3)/3$. The critical cell of $F_\sigma$ has dimension $(2A(\sigma)+2N(\sigma)-3d-N_2(\sigma)-3)/3$.

The numbers $B(\sigma)=(N(\sigma)+N_2(\sigma)+A(\sigma))/3$ and $C(\sigma)=A(\sigma)+N(\sigma)$ will be useful. The number $B(\sigma)$ is an integer as $\sqcup_{i=1}^dP_n[V(\sigma_i)]$ do not have any components with $1\mod 3$ vertices. The dimension of the critical cell in $F_\sigma$ is $C(\sigma)-B(\sigma)-d$, the label of the cell is of degree $C(\sigma)$.

The dimension and label of the critical cell in $F_\sigma$ is determined by the numbers $B(\sigma)$ and $C(\sigma)$, these numbers are determined by the combinatorics of the graph $\sqcup_{i=1}^dP_n[V(\sigma_i)]$. The combinatorics of the graph $\sqcup_{i=1}^dP_n[V(\sigma_i)]$ can be read of from the box diagram of $\sigma$.

Form a graph on the set of boxes in the box diagram of $\sigma$ by letting two boxes on the same row be adjacent in the graph if they are adjacent in the diagram, this graph is isomorphic to the line graph of $\sqcup_{i=1}^dP_n[V(\sigma_i)])$.

To count the number of critical inducing label maximal cells $\sigma$ with given $(B(\sigma),C(\sigma))$ it is convenient to translate the set of these cells into a particular set of $01$-strings. Consider the $d\times n-1$ matrix $M(\sigma)$ obtained from the box diagram of $\sigma$ by replacing every box by a $1$ and every space by a $0$. Let $L(\sigma)$ be the string of length $d(n-1)$ obtained by flattening the matrix $M(\sigma)$.

Let a maximal substring of zeroes in a $01$-string be \emph{interior} if it is surrounded by ones. 

\begin{proposition}\label{prop:string}
The map $\sigma\rightarrow L(\sigma)$ is a bijection between the set of critical inducing label maximal cells $\sigma$ with  $(N,N_2,A)=(N(\sigma),N_2(\sigma),A(\sigma))$ and the set $S$ of $01$-strings of length $d(n-1)$ satisfying:

\begin{itemize}
\item The strings in $S$ have exactly $d-1$ interior maximal substrings of zeroes of length at least $n-2$.
\item The strings in $S$ have exactly $A$ ones.
\item The strings in $S$ have exactly $N$ maximal substrings consisting of ones and of these $N$ substrings exactly $N_2$ consists of $1\mod 3$ ones.
\end{itemize}
\end{proposition}
\begin{proof}
The map is an injection into some set $S$ of $01$-strings as the box diagram can be recovered by inserting $d-1$ line breaks to get the desired matrix that encode the diagram.

The strings obtained from cells $\sigma$ has at least $d-1$ interior maximal strings of length at least $n-2$ as the ending zeroes of row $i$ and beginning zeroes of row $i+1$ in $M(\sigma)$  together gives $d-1$ such subsequences. There can not be more interior strings of this type as $d(n-2)+(d+1)>d(n-1)$.

The box diagram encodes the line graph of $\sqcup_{i=1}^dP_n[V(\sigma_i)]$ and this proves the last two properties of $S$.

To prove that the map is surjective it is enough to check that the matrix recovered by inserting the $d-1$ line breaks encode a valid box diagram.
A string encode a box diagram of each line break touch one of the interior maximal substrings of zeros of length at least $n-2$.

Each maximal interior substring zeroes of length at least $n-2$ touch a line break as the rows have length $n-1$, a line break can not touch two different of these strings and we are done as there are exactly $d-1$ strings of this type.
\end{proof}

\begin{lemma}\label{lemma:countcell3}
The number of cells $\sigma$ with $(N(\sigma),B(\sigma),C(\sigma))=(N,B,C)$ is

\[{N\choose 3B-C}{N-1\choose d-1}{n+3d-C-2\choose N}{B-1\choose N-1}.\]
\end{lemma}
\begin{proof}
The proof goes by proving that this is the size of the corresponding set $S$ in Proposition~\ref{prop:string}. The relations $N_2(\sigma)=3B(\sigma)-C(\sigma)$ and $A(\sigma)=C(\sigma)-N(\sigma)$ are useful in the following argument.

The integer $N$ is the number of maximal substrings consisting of ones in the stings in $S$. The integer $C$ is the total number of ones in the string plus the number of maximal substrings consisting of ones in the strings in $S$. The integer $3B-C$ is the number maximal strings of ones of length $1\mod3$.

The number of ways to distribute which sequences of ones are of length $1\mod3$ is ${N\choose 3B-C}.$
There are $N-1$ interior spaces and $d-1$ of them should have at least length $n-2$. This can be done in ${N-1\choose d-1}$ ways.
There are $n+3d-N-C-2$ remaining zeroes to distribute in $N+1$ subsequences. This can be done in ${n+3d-C-2\choose N}$ ways.
What remains is to distribute the remaining $3B-3N$ ones into the $N$ maximal substrings of ones. To not change the lengths $\mod 3$ this is done by multiples of three, this can be done in ${B-1\choose N-1}$ ways.
\end{proof}

\begin{lemma}\label{lemma:countcell4}
The number of cells $\sigma$ with $(B(\sigma),C(\sigma))=(B,C)$ is
\[{n+3d-C-2\choose 3B-C}  {n+2d-2B-2\choose C-2B}   {B-1\choose d-1}. \]

\end{lemma}
\begin{proof}
This number is obtained by summing the expression in Lemma~\ref{lemma:countcell3} over all possible values of $N$.

First rewrite the expression
\[
\begin{array}{rcl}
{N\choose 3B-C}{N-1\choose d-1}{n+3d-C-2\choose N}{B-1\choose N-1}&=&{n+3d-C-2\choose N}{N\choose 3B-C}         {B-1 \choose N-1} {N-1 \choose d-1}\\
&  = &  {n+3d-C-2 \choose 3B-C}{(n+3d-C-2)-(3B-C)\choose N-(3B-C)} {B-1\choose d-1}{(B-1)-(d-1)\choose (N-1)-(d-1)}  \\
& = &  {n+3d-C-2 \choose 3B-C} {n+3d-3B-2\choose N+C-3B}  {B-1\choose d-1} {B-d\choose N-d}.\\
\end{array}
\]
Using that $\sum_i {a \choose i+j}{b \choose i+k} = {a+b \choose b+j-k}$ it is possible to compute the sum
\[
\begin{array}{rcl}
 \sum_{N} {n+3d-C-2 \choose 3B-C} {n+3d-3B-2\choose N+C-3B}  {B-1\choose d-1} {B-d\choose N-d}
 & = &   {n+3d-C-2 \choose 3B-C}   {B-1\choose d-1}   \sum_{N} {n+3d-3B-2\choose N+C-3B} {B-d\choose N-d} \\
 & = &  {n+3d-C-2 \choose 3B-C}   {B-1\choose d-1}   { (n+3d-3B-2) + (B-d)   \choose (B - d)+(C-3B)-(- d)  }   \\
 & = &  {n+3d-C-2 \choose 3B-C}   {B-1\choose d-1}   { n+2d-2B-2\choose C-2B  }   \\
 & = &  {n+3d-C-2 \choose 3B-C}      { n+2d-2B-2\choose C-2B  }{B-1\choose d-1}.     \\
\end{array}
\]
\end{proof}

\begin{theorem}
The graded Betti number $\beta_{i,j}(S/I^d_{P_n})={n+3d-j-2\choose 2j-3i-3d+3}{n+4d+2i-2j-4\choose 2d+2i-j-2}{j-i-d\choose d-1}$
\end{theorem}
\begin{proof}
Recall that the dimension of the critical cell induced by a label maximal cell is $C(\sigma)-B(\sigma)-d$ and the label is of degree $C(\sigma)$, these cells contribute to the Betti number $\beta_{i,j}(S/I^d_{P_n})$ with $i=C-B-d+1$ and $j=C$. The number of cells with given $B(\sigma)$ and $C(\sigma)$ are counted in Lemma~\ref{lemma:countcell4}.
\end{proof}


\begin{thebibliography}{99}
 \bibitem{Ba}
Carlos Bahiano.
Symbolic powers of edge ideals.
\emph{J. Algebra} {\bf 273} (2004), no. 2, 517--537.  
 
\bibitem{BW}
Ekkehard Batzies and Volkmar Welker.
Discrete Morse theory for cellular resolutions. 
\emph{J. Reine Angew. Math.} {\bf 543} (2002), 147--168.  
  
\bibitem{benedettiXX}
Bruno Benedetti.
Discrete Morse Theory for Manifolds with Boundary.
\emph{Trans. Amer. Math. Soc.} {\bf 364} (2012) 6631--6670.

\bibitem{bjornerButlerMatveev97}
Anders Bj\"orner, Lynne M. Butler and Andrey O. Matveev.
Note on a combinatorial application of Alexander duality.
\emph{J. Combin. Theory Ser. A} {\bf 80} (1997), no. 1, 163--165. 

\bibitem{B}
Markus Brodmann.
The asymptotic nature of the analytic spread. 
\emph{Math. Proc. Cambridge Philos. Soc.} {\bf 86} (1979), no. 1, 35--39. 

\bibitem{dochtermannEngstrom2012}
Anton Dochtermann and Alexander Engstr\"om.
Cellular resolutions of cointerval ideals. 
\emph{Math. Z.} {\bf 270} (2012), no. 1-2, 145--163. 

\bibitem{engstrom09a}
Alexander  Engstr\"om.
Complexes of directed trees and independence complexes.
\emph{Discrete Math.} {\bf 309} (2009), no. 10, 3299--3309. 

\bibitem{forman98}
Robin Forman.
Morse theory for cell complexes. 
\emph{Adv. Math.} {\bf 134} (1998), no. 1, 90--145. 

\bibitem{HH}
J\"urgen Herzog and Takayuki Hibi.
The depth of powers of an ideal.
\emph{J. Algebra} {\bf 291} (2005), no. 2, 534--550. 

\bibitem{HW}
J\"urgen Herzog and Volkmar Welker
The Betti polynomials of powers of an ideal.
\emph{J. Pure Appl. Algebra} {\bf 215} (2011), no. 4, 589--596. 

\bibitem{hibi}
Takayuki Hibi and Hidefumi Ohsugi.
Normal polytopes arising from finite graphs.
\emph{J. Algebra} {\bf 207} (1998), no. 2, 409--426. 


\bibitem{jonsson08}
Jakob Jonsson.
\emph{Simplicial complexes of graphs.} 
Lecture Notes in Mathematics, 1928. Springer-Verlag, Berlin, 2008. xiv+378 pp.

\bibitem{kawamura11}
Kazuhiro Kawamura.
Independence complexes and edge covering complexes via Alexander duality.
\emph{Electron. J. Combin.} {\bf 18} (2011), no. 1, Paper 39, 6 pp.

\bibitem{K}
Vijay Kodiyalam.
Homological invariants of powers of an ideal. 
\emph{Proc. Amer. Math. Soc.} {\bf 118} (1993), no. 3, 757--764. 

\bibitem{deLoeraEtAlBook}
Jes\'us De Loera, J\"org Rambau and Francisco Santos.
\emph{Triangulations.} 
Algorithms and Computation in Mathematics, 25. Springer-Verlag, Berlin, 2010. 535 pp.

\bibitem{MS}
Ezra Miller and Bernd Sturmfels.
\emph{Combinatorial commutative algebra.} 
Graduate Texts in Mathematics, 227. Springer-Verlag, New York, 2005. 417 pp.

\bibitem{M}
Susan Morey.
Depths of powers of the edge ideal of a tree. 
\emph{Comm. Algebra} {\bf 38} (2010), no. 11, 4042--4055.

\bibitem{schrijver86}
Alexander Schrijver.
\emph{Theory of linear and integer programming.} 
Wiley-Interscience Series in Discrete Mathematics. A Wiley-Interscience Publication. John Wiley \& Sons, Ltd., Chichester, 1986. 471 pp.

\bibitem{SVV}
Aron Simis, Wolmer Vasconcelos and Rafael Villarreal.
On the ideal theory of graphs.
\emph{J. Algebra } {\bf 167} (1994), no. 2, 389--416. 








\end{thebibliography}
\end{document}